\documentclass{amsart}
\usepackage{amsmath, amsthm, amscd, amsfonts, amssymb, graphicx, color}
\usepackage[bookmarksnumbered, colorlinks, plainpages]{hyperref}
\newtheorem{theorem}{Theorem}[section]
\newtheorem{lemma}[theorem]{Lemma}

\theoremstyle{definition}

\newtheorem{example}[theorem]{Example}

\theoremstyle{remark}

\numberwithin{equation}{section}
\newcommand{\Mmn}{\mathbb{M}_m \to \mathbb{M}_n}
\newcommand{\tr}{\operatorname{tr}}
\newcommand{\HS}{\text{HS}}
\newcommand{\dcp}{d_{\mathrm{CP}}}
\begin{document}

\title{Hermitian Maps:   Approximations  and Completely Positive Extensions}
\author[M. Kian, M.Rostamian Delavar]{Mohsen Kian and Mohsen Rostamian Delavar}
\address{Mohsen Kian: Department of Mathematics, University of Bojnord, P. O. Box 1339, Bojnord 94531, Iran}
\email{kian@ub.ac.ir }
\address{Mohsen Rostamian Delavar: Department of Mathematics, University of Bojnord, P. O. Box 1339, Bojnord 94531, Iran}
\email{a.rostamian@ub.ac.ir}

\subjclass[2020]{Primary:47A65, 15A69; Secondary: 15A60}

\keywords{Hermitian map,   completely positive map, Choi matrxi,  decomposition, CP-distance}

\begin{abstract}
This study investigates Hermitian linear maps, focusing on their decomposition into completely positive (CP) maps and their extensions to CP maps using auxiliary spaces. We derive a precise lower bound on the Hilbert-Schmidt norm of the negative component in any CP decomposition, proving its attainability through the Jordan decomposition. Additionally, we demonstrate that the positive part of this decomposition provides the optimal CP approximation in the Hilbert-Schmidt norm. We also determine the minimal dimension of an auxiliary space required to extend a Hermitian map to a CP map, with explicit constructions provided. Practical examples illustrate the application of our results
\end{abstract}

\maketitle

\section{Introduction and Preliminaries}

Linear maps between matrix algebras are pivotal in quantum information theory and operator algebras, serving as the foundation for understanding quantum operations and their structural properties. Hermitian linear maps, which preserve the Hermitian property of matrices, and completely positive (CP) maps, which model physically realizable quantum channels, hold particular importance. These maps are central to describing quantum dynamics, entanglement transformations, and information processing protocols. This paper elucidates the interplay between Hermitian and CP maps, with a focus on their decomposition into CP maps, optimal approximations in the Hilbert-Schmidt norm, and extensions to CP maps via auxiliary spaces. CP maps, characterized by positive semi-definite Choi matrices, are indispensable in quantum mechanics, ensuring positivity across all tensor extensions \cite{watrous2018}. In contrast, Hermitian maps, while preserving Hermitian properties, may lack complete positivity, leading to deviations from physical realizability. Quantifying these deviations and constructing CP approximations or extensions are essential for applications such as quantum channel correction, error mitigation, and quantum algorithm design.

The study of positive and CP maps forms a cornerstone of operator algebras and quantum information theory. St{\o}rmer’s seminal work \cite{stormer2013} offers a comprehensive framework for positive linear maps on C*-algebras, elucidating their algebraic and structural properties.

Recent advances have deepened our understanding of positive map decomposition and approximation. Ando \cite{ando2018} investigated the decomposition of positive maps as the difference of two CP or super-positive maps, leveraging the celebrated Choi matrix representation. This concept has been extended to infinite-dimensional settings by Han, Kye, and St{\o}rmer \cite{han2024}, providing a broader framework for linear maps in quantum systems. In realm of quantum information, Dadkhah, Kian, and Moslehian \cite{dadkhah2024} explored the decomposition of tracial positive maps, highlighting their utility in quantum information processing. The first author’s recent work \cite{kian2025} analyzed deviations from complete positivity, offering structural insights into Hermitian maps and their approximations by leveraging Choi matrix eigenvalues. This work underpins our optimal approximation results and motivates our use of the Hilbert-Schmidt norm to quantify deviations. Positive maps, including Hermitian-preserving ones, are crucial for entanglement detection, with recent studies introducing partially contractive maps to refine the classification of entangled states \cite{sisu2018}. Additionally, Buscemi et al. \cite{buscemi2013} proposed a universal two-point quantum correlator, decomposing Hermitian maps into CP maps to measure correlation functions, inspiring our statistical decomposition approach.

The concept of CP extensions is closely linked to quantum broadcasting and dilation theorems. Parzygnat, Fullwood, Buscemi, and Chiribella \cite{parzygnat2023} introduced virtual quantum broadcasting, which extends quantum operations to higher-dimensional spaces while preserving physical properties.

Our work advances the study of Hermitian maps by establishing a sharp lower bound on the Hilbert-Schmidt norm of the negative component in their decomposition into completely positive (CP) maps, proving that the positive part of the Jordan decomposition yields the optimal CP approximation. Additionally, we determine the minimal dimension of an auxiliary space required to extend a Hermitian map to a CP map, providing explicit constructions to support practical applications.

\subsection*{Notations and Preliminaries}

We denote by \(\mathbb{M}_m\) the algebra of all \(m \times m\) complex matrices, with \(I_m\) representing the \(m \times m\) identity matrix. The L\"{o}wner partial order on Hermitian matrices is defined as \(A \leq B\) if \(B - A\) is positive semi-definite. For a Hermitian matrix \(A \in \mathbb{M}_m\), the Jordan decomposition is \(A = A^+ - A^-\), where \(A^+ = \frac{1}{2}(A + |A|)\) and \(A^- = \frac{1}{2}(|A| - A)\) are the positive and negative parts, respectively, with \(|A| = (A^* A)^{1/2}\).

For a Hermitian matrix \(A \in \mathbb{M}_{m}\), the multiplicity of an eigenvalue \(\lambda\) is the dimension of the corresponding eigenspace, i.e., the dimension of \(\ker(A - \lambda I_{m})\).

 The Choi matrix of a linear map \(\Phi:\Mmn\) is defined as
\[
\mathbf{C}_\Phi = \sum_{i,j=1}^m E_{ij} \otimes \Phi(E_{ij}) \in \mathbb{M}_m \otimes \mathbb{M}_n,
\]
where \(\{E_{ij}\}\) is  the standard matrix units in $\mathbb{M}_m$. The linear map $\Phi$ can be recovered from its Choi matrix $\mathbf{C}_\Phi$ via
\[
\Phi(X) = \sum_{j,k=1}^m \xi_{jk} \Phi(E_{jk}), \quad \text{where } \xi_{jk} = \langle X, E_{kj} \rangle = \operatorname{tr}(X E_{kj}^*).
\]
 This defines an isomorphism, known as the Choi–Jamio{\l}kowski isomorphism.  The Choi matrix representation, fundamental to our analysis, has been recently revisited in the context of positive maps in quantum information \cite{Majewski2024}. Furthermore, the Choi–Jamio{\l}kowski isomorphism, as detailed in recent work \cite{homa2024}, provides a practical framework for analyzing completely positive maps in quantum information.
 
 The Hilbert-Schmidt norm of \(\Phi\) is given by
\[
\|\Phi\|_{HS} = \sqrt{\tr(\mathbf{C}_\Phi^* \mathbf{C}_\Phi)}.
\]

The CP-distance of linear mapping  \(\Phi: \mathbb{M}_m \to \mathbb{M}_n\) is  defined by $d_{\mathrm{CP}}(\Phi)=\max\{0,-\lambda_{\min}(\mathbf{C}_\Phi)\}$. It was introduced
 in \cite{kian2025} as a measure of deviation $\Phi$ from being complete positive.

A linear map \(\Phi: \mathbb{M}_m \to \mathbb{M}_n\) is called \emph{Hermitian} (or Hermitian-preserving) if \(\Phi(A^*) = \Phi(A)^*\) for all \(A \in \mathbb{M}_m\). It    is called \emph{positive} if it maps positive semi-definite matrices of $\mathbb{M}_m$    to positive semi-definite matrices in \(\mathbb{M}_n\). A map $\Phi$ is \emph{completely positive} (CP) if it remains positive under all tensor extensions, i.e., \(I_k \otimes \Phi: \mathbb{M}_k \otimes \mathbb{M}_m \to \mathbb{M}_k \otimes \mathbb{M}_n\) is positive for all \(k\), where \(I_k\) is the identity map on \(\mathbb{M}_k\).
It is known that $\Phi$   is    completely positive  if and only if its Choi matrix $\mathbf{C}_\Phi$ is a positive semi-definite matrix \cite{choi1975}.

For \(X \in \mathbb{M}_m \otimes \mathbb{M}_n\), the partial trace over the second factor is the linear map \(\tr_2: \mathbb{M}_m \otimes \mathbb{M}_n \to \mathbb{M}_m\) defined by
\[
\tr_2(X) = \sum_{k=1}^n (I_m \otimes E_{kk}) X (I_m \otimes E_{kk}),
\]
where \(E_{kk} \in \mathbb{M}_n\) are the standard matrix units with 1 in the \((k,k)\)-entry and 0 elsewhere, and \(I_m\) is the identity in \(\mathbb{M}_m\).
For a tensor product \(A \in \mathbb{M}_m\), \(B \in \mathbb{M}_n\), the partial trace satisfies
\[
\tr_2(A \otimes B) = \tr(B) A.
\]
Similarly, the partial trace over the first factor, \(\tr_1: \mathbb{M}_m \otimes \mathbb{M}_n \to \mathbb{M}_n\), is defined by
\[
\tr_1(X) = \sum_{j=1}^m (E_{jj} \otimes I_n) X (E_{jj} \otimes I_n),
\]
with \([\tr_1(X)]_{kl} = \sum_{j=1}^m x_{jk,jl}\) and \(\tr_1(A \otimes B) = \tr(A) B\). The partial trace is independent of the choice of basis and satisfies \(\tr(\tr_1(X)) = \tr(\tr_2(X)) = \tr(X)\).


 \bigskip

\section{Main Results}
We demonstrate that in any decomposition of a Hermitian map into two completely positive maps, the magnitude of the negative component depends on both the extent of deviation from complete positivity and the dimension of the associated eigenspace. To achieve  this, we need a lemma, which, despite appearing to be a known result in matrix analysis, lacks a direct reference in the literature, prompting us to include a proof.

\begin{lemma}\label{lm-1}
Let $A=A^+-A^-$ be the Jordan decomposition of a Hermitian matrix $A\in\mathbb{M}_d$. Then
  $ A^- $ is the minimal positive semi-definite matrix such that $ A + A^- \geq 0 $. This mean that any $ B \geq 0 $ with $ A + B \geq 0 $ satisfies $ B \geq A^- $, and $ A^- $ itself achieves equality in this bound.
\end{lemma}
\begin{proof}
 The spectral theorem allows us to write:
$ A = \sum_i \lambda_i P_i, $
where $ \lambda_i $ are the real eigenvalues (positive, negative, or zero), and $ P_i $ are the orthogonal projection   onto the corresponding eigenspaces, satisfying $ P_i P_j = \delta_{ij} P_i $ and $ \sum_i P_i = I $.
The Jordan decomposition defines the positive part and negative part as  $ A^+ = \sum_{\lambda_i > 0} \lambda_i P_i $ and  $ A^- = \sum_{\lambda_i < 0} (-\lambda_i) P_i $, respectively.
Define subspaces
\begin{align*}
   V^+ &= \bigoplus_{\lambda_i > 0} \operatorname{im}(P_i): \quad \mbox{ subspace for positive eigenvalues};\\
  V^- &= \bigoplus_{\lambda_i < 0} \operatorname{im}(P_i):  \quad \mbox{subspace for negative eigenvalues};\\
  V^0 &= \operatorname{ker}(A) = \bigoplus_{\lambda_i = 0} \operatorname{im}(P_i): \quad \mbox{kernel of $ A $}.
\end{align*}
These subspaces are orthogonal, and the Hilbert space decomposes as $ \mathbb{C}^n = V^+ \oplus V^- \oplus V^0 $.

On $ V^- $, we have $ A = \sum_{\lambda_i < 0} \lambda_i P_i = -A^- $.
On $ V^+ $, we have $ A = A^+ $, and $ A^- = 0 $.
On $ V^0 $, we have $ A = 0 $, $ A^- = 0 $.

Now suppose  $ B \geq 0 $ and $ A + B \geq 0 $. We need $ B \geq A^- $.
For any $ x = x^+ + x^- + x^0 $ (decomposed along $ V^+ \oplus V^- \oplus V^0 $),
$$ \langle x, (B - A^-) x \rangle = \langle x^+, B x^+ \rangle + \langle x^-, (B - A^-) x^- \rangle + \langle x^0, B x^0 \rangle, $$
 as $A^{-}=0$ on $ V^+ $ and $ V^0$.
 Each term is non-negative, specifically, since $x_-\in V^{-}$ and $A=-A^{-}$ on $V^{-}$, the hypothsis $A+B\geq0$ implies $B\geq A^{-}$ on $V^{-}$.  This complete the proof.
 \end{proof}

We propose  a lower bound on the Hilbert-Schmidt norm of the negative component in \emph{any} decomposition of a Hermitian map into two completely positive maps.
\begin{theorem}
\label{thm:main}
Let \(\Phi: \Mmn\) be a Hermitian linear map with the Choi matrix $\mathbf{C}_\Phi$.    Let \(k\) be the multiplicity of the smallest eigenvalue \(\lambda_{\min}(\mathbf{C}_\Phi) = -\dcp(\Phi)\) of  \(\mathbf{C}_\Phi\). Then, in any decomposition \(\Phi = \Phi^{(1)} - \Phi^{(2)}\), where \(\Phi^{(1)}\) and \(\Phi^{(2)}\) are completely positive maps, the following inequality holds:
\begin{align}\label{main1}
\|\Phi^{(2)}\|_{\HS} \geq \sqrt{k}\,\dcp(\Phi).
\end{align}
Moreover, this bound is sharp, meaning there exists a decomposition where equality is achieved.
\end{theorem}

\begin{proof}
Since \(\Phi\) is a Hermitian linear map with \(\dcp(\Phi) > 0\), the Choi matrix \(\mathbf{C}_\Phi \in \mathbb{M}_m \otimes \mathbb{M}_n \) is Hermitian but not positive semi-definite, and its smallest eigenvalue satisfies \(\lambda_{\min}(\mathbf{C}_\Phi) = -\dcp(\Phi) < 0\) with multiplicity \(k\).
Consider the spectral decomposition of \(\mathbf{C}_\Phi\):
\[
\mathbf{C}_\Phi = \sum_{i} \lambda_i P_i,
\]
where \(\lambda_i\) are the  eigenvalues of \(\mathbf{C}_\Phi\), and \(P_i\) are the orthogonal projections onto the corresponding eigenspaces, satisfying \(\sum_i P_i = I_{mn}\) and \(P_i P_j = \delta_{ij} P_i\). Since \(\mathbf{C}_\Phi\) is Hermitian, all \(\lambda_i \in \mathbb{R}\). Define the positive and negative parts:
\[
\mathbf{C}_\Phi^+ = \sum_{\lambda_i > 0} \lambda_i P_i, \quad \mathbf{C}_\Phi^- = \sum_{\lambda_i < 0} (-\lambda_i) P_i,
\]
so that
\[
\mathbf{C}_\Phi = \mathbf{C}_\Phi^+ - \mathbf{C}_\Phi^-,
\]
is the  Jordan decomposition, where \(\mathbf{C}_\Phi^+ \geq 0\), \(\mathbf{C}_\Phi^- \geq 0\), and \(\mathbf{C}_\Phi^+ \mathbf{C}_\Phi^- = 0\) due to the orthogonality of the projections.  Let \(\Phi^+\) and \(\Phi^-\) be the completely positive maps with Choi matrices \(\mathbf{C}_\Phi^+\) and \(\mathbf{C}_\Phi^-\), respectively, so that
\[
\Phi = \Phi^+ - \Phi^-.
\]
The Hilbert-Schmidt norm of \(\Phi^-\) is given by
\[
\|\Phi^-\|_{\HS} = \sqrt{\tr(\mathbf{C}_\Phi^- \mathbf{C}_\Phi^-)} = \sqrt{\tr((\mathbf{C}_\Phi^-)^2)}.
\]
Since
\[
\mathbf{C}_\Phi^- = \sum_{\lambda_i < 0} (-\lambda_i) P_i,
\]
we compute
\[
(\mathbf{C}_\Phi^-)^2 = \left( \sum_{\lambda_i < 0} (-\lambda_i) P_i \right)^2 = \sum_{\lambda_i < 0} (-\lambda_i)^2 P_i,
\]
because \(P_i P_j = \delta_{ij} P_i\). Thus,
\[
\tr((\mathbf{C}_\Phi^-)^2) = \tr\left( \sum_{\lambda_i < 0} (-\lambda_i)^2 P_i \right) = \sum_{\lambda_i < 0} (-\lambda_i)^2 \tr(P_i).
\]
Let \(\lambda_{\min} = -\dcp(\Phi)\) be the only negative  eigenvalue, with corresponding projection \(P_{\min}\) and multiplicity \(\tr(P_{\min}) = k\).   Then:
\[
\mathbf{C}_\Phi^- = (-\lambda_{\min}) P_{\min} = \dcp(\Phi) P_{\min}\quad\Rightarrow\quad (\mathbf{C}_\Phi^-)^2 = \dcp(\Phi)^2 P_{\min}^2 = \dcp(\Phi)^2 P_{\min},
\]
and so
\[
\tr((\mathbf{C}_\Phi^-)^2) = \dcp(\Phi)^2 \tr(P_{\min}) = \dcp(\Phi)^2 \cdot k.
\]
So:
\[
\|\Phi^-\|_{\HS} =\sqrt{\tr((\mathbf{C}_\Phi^-)^2)} =\sqrt{\dcp(\Phi)^2 \cdot k} = \dcp(\Phi) \cdot \sqrt{k}.
\]
This shows that in particular the bound in \eqref{main1} is sharp. This is somehow the  size of \(\Phi^-\) in the decomposition \(\Phi = \Phi^+ - \Phi^-\), in which \(\lambda_{\min}\) is the only negative  eigenvalue.  If there are other negative eigenvalues, \(\tr((C_\Phi^-)^2)\) could be larger, say
\begin{align*}
\|\Phi^-\|_{\HS}^2=\tr((\mathbf{C}_\Phi^-)^2)&= \sum_{\lambda_i < 0} (-\lambda_i)^2 \tr(P_i)\\
&=(-\lambda_{\mathrm{min}})^2\tr(P_{\min})+\sum_{\lambda_i < 0,\lambda_i\neq\lambda_{\mathrm{min}}} (-\lambda_i)^2 \tr(P_i)\\
&\geq (-\lambda_{\mathrm{min}})^2\tr(P_{\min})=\dcp(\Phi)^2 \cdot k.
\end{align*}
This implies the desired inequality \eqref{main1} in the case where $\Phi^{(1)}=\Phi^+$ and $\Phi^{(2)}=\Phi^-$.

Now consider an arbitrary decomposition \(\Phi = \Phi^{(1)} - \Phi^{(2)}\), where \(\Phi^{(1)}\) and \(\Phi^{(2)}\) are completely positive maps with Choi matrices \(\mathbf{C}_{\Phi^{(1)}} \geq 0\) and \(\mathbf{C}_{\Phi^{(2)}} \geq 0\) and $\mathbf{C}_\Phi=\mathbf{C}_{\Phi^{(1)}}-\mathbf{C}_{\Phi^{(2)}}$.
It follows from Lemma~\ref{lm-1} that $\mathbf{C}_{\Phi}^{-}$ is the minimal positive semi-definite matrix satisfying $\mathbf{C}_{\Phi}+\mathbf{C}_{\Phi}^{-}\geq0$. Since $\mathbf{C}_{\Phi^{(2)}}\geq0$ and $\mathbf{C}_{\Phi}+\mathbf{C}_{\Phi^{(2)}}\geq0$, this implies that $\mathbf{C}_{\Phi^{(2)}}\geq \mathbf{C}_{\Phi}^{-}$. Consequently,
$$\|\Phi^{(2)}\|_{\HS}\geq \|\Phi^{-}\|_{\HS}\geq\dcp(\Phi) \cdot \sqrt{k} $$
as required.
\end{proof}

To provide an examples   illustrating   Theorem~\ref{thm:main}, we consider the transposition map, a well-known positive but not completely positive map, that can be analyzed using recent tools for constructing positive maps in matrix algebras \cite{Piotrowski2023}.

\begin{example}[Transposition Map]
Consider the transposition map \(\Phi: \mathbb{M}_2 \to \mathbb{M}_2\) defined by \(\Phi(A) = A^T\), where \(A^T\) is the transpose of \(A\). The Choi matrix is
\[
\mathbf{C}_\Phi = \sum_{i,j=1}^2 E_{ij} \otimes E_{ji} = \begin{bmatrix}
1 & 0 & 0 & 0 \\
0 & 0 & 1 & 0 \\
0 & 1 & 0 & 0 \\
0 & 0 & 0 & 1
\end{bmatrix},
\]
which is the swap operator in \(\mathbb{M}_2 \otimes \mathbb{M}_2\). Its eigenvalues are 1 (multiplicity 3) and \(-1\) (multiplicity 1), so
\[
\lambda_{\min}(\mathbf{C}_\Phi) = -1, \quad \dcp(\Phi) = 1, \quad k = 1.
\]
In the Jordan decomposition, \(\mathbf{C}_\Phi^- = P_{-}\), the projection onto the antisymmetric subspace, given by
\[
P_{-} = \frac{1}{2} (I_4 - \mathbf{C}_\Phi) = \begin{bmatrix}
0 & 0 & 0 & 0 \\
0 & 1 & -1 & 0 \\
0 & -1 & 1 & 0 \\
0 & 0 & 0 & 0
\end{bmatrix}.
\]
The Hilbert-Schmidt norm of \(\Phi^-\) is
\[
\|\Phi^-\|_{\HS} = \sqrt{\tr(P_{-}^2)} = \sqrt{\tr(P_{-})} = \sqrt{1} = 1.
\]
The bound is
\[
\dcp(\Phi) \cdot \sqrt{k} = 1 \cdot \sqrt{1} = 1,
\]
which is achieved, confirming sharpness. For any other decomposition \(\Phi = \Phi^{(1)} - \Phi^{(2)}\), such as \(\Phi^{(1)}(A) = A^T + \tr(A) I_2\), \(\Phi^{(2)}(A) = \tr(A) I_2\), we compute \(\|\Phi^{(2)}\|_{\HS} = \sqrt{4} = 2 > 1\), satisfying the bound.
\end{example}

\begin{example}[Map with Higher Multiplicity]
Define \(\Phi: \mathbb{M}_2 \to \mathbb{M}_2\) by
\[
\Phi(A) = \begin{bmatrix}
a_{11} + a_{22} & -a_{12} \\
-a_{21} & a_{11} + a_{22}
\end{bmatrix},
\]
where \(A = [a_{ij}]\). The Choi matrix is computed as
\[
\mathbf{C}_\Phi = \begin{bmatrix}
2 & 0 & 0 & 0 \\
0 & 0 & -1 & 0 \\
0 & -1 & 0 & 0 \\
0 & 0 & 0 & 2
\end{bmatrix}.
\]
The eigenvalues are 2 (multiplicity 2), \(-1\) (multiplicity 2), so
\[
\lambda_{\min}(\mathbf{C}_\Phi) = -1, \quad \dcp(\Phi) = 1, \quad k = 2.
\]
The negative part is
\[
\mathbf{C}_\Phi^- = \begin{bmatrix}
0 & 0 & 0 & 0 \\
0 & 1 & -1 & 0 \\
0 & -1 & 1 & 0 \\
0 & 0 & 0 & 0
\end{bmatrix},
\]
with
\[
\|\Phi^-\|_{\HS} = \sqrt{\tr((\mathbf{C}_\Phi^-)^2)} = \sqrt{\tr(\mathbf{C}_\Phi^-)} = \sqrt{2}.
\]
The bound is
\[
\dcp(\Phi) \cdot \sqrt{k} = 1 \cdot \sqrt{2} = \sqrt{2},
\]
achieved in the Jordan decomposition, confirming sharpness.
\end{example}
A key question is how to approximate a Hermitian linear map with a completely positive one in a meaningful way. The next theorem   identifies the best such approximation in the Hilbert-Schmidt norm.

\begin{theorem}\label{thm:main2}
Let \(\Phi: \mathbb{M}_m \to \mathbb{M}_n\) be a Hermitian linear map with Jordan decomposition \(\Phi = \Phi^+ - \Phi^-\), where \(\Phi^+\) and \(\Phi^-\) are completely positive maps corresponding to the positive and negative parts of the Choi matrix \(\mathbf{C}_\Phi\), respectively. Then, the best approximation of \(\Phi\) by a completely positive map \(\Psi\) in the Hilbert-Schmidt norm is \(\Psi = \Phi^+\), and:
\begin{align}\label{th2-q1}
\min_{\Psi \text{ CP}} \|\Phi - \Psi\|_{HS} = \|\Phi - \Phi^+\|_{HS} = \|\Phi^-\|_{HS}.
\end{align}
\end{theorem}
\begin{proof}
Consider \(\Psi = \Phi^+\), which is completely positive. Then $\Phi - \Phi^+ = -\Phi^-$ and
\[
\|\Phi - \Phi^+\|_{HS} = \|-\Phi^-\|_{HS} = \|\Phi^-\|_{HS},
\]
meaning that \(\Psi = \Phi^+\) satisfies \eqref{th2-q1}. Now, for any completely positive map \(\Psi\) with \(\mathbf{C}_\Psi \geq 0\), we aim to show:
$
\|\Phi - \Psi\|_{HS} \geq \|\Phi^-\|_{HS}.
$
Since  \(\mathbf{C}_\Phi = \mathbf{C}_\Phi^+ - \mathbf{C}_\Phi^-\), we have
\begin{align}\label{th2-q2}
\operatorname{tr}(\mathbf{C}_\Phi \mathbf{C}_\Psi) = \operatorname{tr}( (\mathbf{C}_\Phi^+ - \mathbf{C}_\Phi^-) \mathbf{C}_\Psi ) = \operatorname{tr}(\mathbf{C}_\Phi^+ \mathbf{C}_\Psi) - \operatorname{tr}(\mathbf{C}_\Phi^- \mathbf{C}_\Psi).
\end{align}
Moreover, we conclude  from \(\mathbf{C}_\Phi^+ \mathbf{C}_\Phi^- = 0\) that
\[
\mathbf{C}_\Phi^2 = (\mathbf{C}_\Phi^+ - \mathbf{C}_\Phi^-)^2 = (\mathbf{C}_\Phi^+)^2 + (\mathbf{C}_\Phi^-)^2,
\]
and so
\begin{align}\label{th2-q3}
\operatorname{tr}(\mathbf{C}_\Phi^2) = \operatorname{tr}( (\mathbf{C}_\Phi^+)^2 ) + \operatorname{tr}( (\mathbf{C}_\Phi^-)^2 ).
\end{align}
Now, computing the Hilbert-Schmit norm we have
\begin{align*}
\|\Phi - \Psi\|_{HS}^2 &= \operatorname{tr}\big( (\mathbf{C}_\Phi - \mathbf{C}_\Psi)^2 \big)\\
 &= \operatorname{tr}(\mathbf{C}_\Phi^2) + \operatorname{tr}(\mathbf{C}_\Psi^2) - 2 \operatorname{tr}(\mathbf{C}_\Phi \mathbf{C}_\Psi)\\
&=\operatorname{tr}( (\mathbf{C}_\Phi^+)^2 ) + \operatorname{tr}( (\mathbf{C}_\Phi^-)^2 )+ \operatorname{tr}(\mathbf{C}_\Psi^2)- 2 \operatorname{tr}(\mathbf{C}_\Phi^+ \mathbf{C}_\Psi) + 2 \operatorname{tr}(\mathbf{C}_\Phi^- \mathbf{C}_\Psi)\\
&\qquad\qquad\qquad\qquad\qquad\qquad(\mbox{by \eqref{th2-q2} and \eqref{th2-q3}})\\
&=\operatorname{tr}\big( (\mathbf{C}_\Phi^+ - \mathbf{C}_\Psi)^2 \big)+\operatorname{tr}( (\mathbf{C}_\Phi^-)^2 )+ 2 \operatorname{tr}(\mathbf{C}_\Phi^- \mathbf{C}_\Psi).
\end{align*}
Both \(\mathbf{C}_\Phi^- \geq 0\) and \(\mathbf{C}_\Psi \geq 0\), and so \(\operatorname{tr}(\mathbf{C}_\Phi^- \mathbf{C}_\Psi) \geq 0\). Moreover, as the trace of a square is non-negative, we have $\operatorname{tr}\big( (\mathbf{C}_\Phi^+ - \mathbf{C}_\Psi)^2 \big)\geq0$.
Hence
\[
\|\Phi - \Psi\|_{HS}^2 \geq \operatorname{tr}( (\mathbf{C}_\Phi^-)^2 ) = \|\mathbf{C}_\Phi^-\|_{HS}^2 = \|\Phi^-\|_{HS}^2.
\]

Equality holds when:

\[
\operatorname{tr}\big( (\mathbf{C}_\Phi^+ - \mathbf{C}_\Psi)^2 \big) = 0 \quad \text{and} \quad \operatorname{tr}(\mathbf{C}_\Phi^- \mathbf{C}_\Psi) = 0.
\]

The first implies \(\mathbf{C}_\Phi^+ = \mathbf{C}_\Psi\), i.e., \(\Psi = \Phi^+\). The second holds because \(\mathbf{C}_\Phi^- \mathbf{C}_\Phi^+ = 0\), confirming that \(\Psi = \Phi^+\) achieves the minimum.

Thus, \(\Psi = \Phi^+\) is the best approximation, with minimal distance \(\|\Phi^-\|_{HS}\).
\end{proof}

\begin{example}
Assume that \(\Phi: \mathbb{M}_2 \to \mathbb{M}_2\) is defined by $\Phi(A) = k \operatorname{tr}(A) I$,
where \(k < 0\). Then
\[
\mathbf{C}_\Phi = \sum_{i,j=1}^2 E_{ij} \otimes \Phi(E_{ij}) = k \sum_{i=1}^2 E_{ii} \otimes I = k (E_{11} + E_{22}) \otimes I = k I_2 \otimes I_2 = k I_4.
\]
Since \(k < 0\), \(\mathbf{C}_\Phi\) is negative definite and we have
\[
\mathbf{C}_\Phi^+ = 0, \quad \mathbf{C}_\Phi^- = -k I_4,
\]
since \(-k > 0\). Thus, \(\Phi^+ = 0\) (the zero map), and \(\Phi^- = -\Phi\), where \(\Phi^-(A) = (-k) \operatorname{tr}(A) I\), which is completely positive.
\[
\|\Phi - \Phi^+\|_{HS} = \|\Phi - 0\|_{HS} = \|k I_4\|_{HS} = |k| \sqrt{\operatorname{tr}(I_4)} = 2 |k|,
\]

\[
\|\Phi^-\|_{HS} = \|-k I_4\|_{HS} = (-k) \sqrt{4} = 2 |k|,
\]

which matches.
\end{example}


Having established the optimal approximation of Hermitian maps by completely positive maps in the Hilbert-Schmidt norm, we now address the problem of extending a Hermitian linear map to a completely positive map using an auxiliary space. This question, motivated by quantum broadcasting and dilation techniques \cite{parzygnat2023}, is critical for applications in quantum information processing, such as channel correction and entanglement manipulation. We examine the minimal dimension of the auxiliary space required for such an extension, providing an explicit construction that leverages the spectral properties of the Choi matrix.

To prevent ambiguity, we use $\tr_m(\cdot)$ to denote the partial trace over the space $\mathbb{M}_m$.
\begin{theorem}\label{th3:main}
Let \(\Phi: \mathbb{M}_m \rightarrow \mathbb{M}_n\) be a Hermitian  linear map with Choi matrix \(\mathbf{C}_{\Phi} \in \mathbb{M}_m \otimes \mathbb{M}_n\). There exists an integer \(k=k_{\min}\), a Hermitian matrix \(Q \in \mathbb{M}_k\),  and a completely positive map \(\Psi: \mathbb{M}_m \otimes \mathbb{M}_k \rightarrow \mathbb{M}_n \otimes \mathbb{M}_k\) extending $\Phi$ in the sense that
\begin{align}\label{rep}
\Phi(X) = \mathrm{tr}_k \left[ \Psi(X \otimes I_k) (I_n \otimes Q) \right] \quad \forall X \in \mathbb{M}_m.
\end{align}
 The minimal dimension \(k_{\min}\) equals the rank of \(\mathbf{C}_{\Phi}\).
\end{theorem}
\begin{proof}
 Since \(\Phi\) is a  Hermitian map, \(\mathbf{C}_{\Phi}\) is a Hermitian matrix. Compute its spectral decomposition:
\[
\mathbf{C}_{\Phi} = \sum_{i=1}^r \lambda_i u_i u_i^*,
\]
where \(\lambda_i \in \mathbb{R}\), \(\{u_i\}_{i=1}^r\) are orthonormal vectors in \(\mathbb{C}^{m n}\), and \(r = \text{rank}(\mathbf{C}_{\Phi})\).

Choose the dimension of the auxiliary space as \(k = r\), so the auxiliary space is \(\mathbb{C}^k\), and let \(\{e_i\}_{i=1}^r\) be an orthonormal basis for \(\mathbb{C}^k\).

We define  \(\Psi: \mathbb{M}_m \otimes \mathbb{M}_k \rightarrow \mathbb{M}_n \otimes \mathbb{M}_k\)  by
\[
\Psi(Y) = \sum_{i=1}^r |\lambda_i| (A_i \otimes e_i) Y (A_i^* \otimes e_i^*)\qquad Y\in\mathbb{M}_m \otimes \mathbb{M}_k,
\]
where each \(A_i: \mathbb{C}^m \rightarrow \mathbb{C}^n\) is a linear operator derived from \(u_i\) via the isomorphism \(\mathbb{C}^m \otimes \mathbb{C}^n \cong \mathbb{M}_{n \times m}\). In other words, $\mathrm{vec}(A_i)=u_i$ for $i=1,\dots,r$.  It is evident that $\Phi$ is completely positive.
Now we define \(Q \in \mathbb{M}_k\) by
\[
Q = \sum_{i=1}^r \text{sgn}(\lambda_i) e_i e_i^*,
\]
where \(\text{sgn}(\lambda_i) = 1\) if \(\lambda_i > 0\), \(-1\) if \(\lambda_i < 0\), and 0 if \(\lambda_i = 0\). Since \(\text{sgn}(\lambda_i) \in \mathbb{R}\), \(Q = Q^*\), so \(Q\) is Hermitian.
For every  \(X \in \mathbb{M}_m\) we have
\[
\Psi(X \otimes I_k) = \sum_{i=1}^r |\lambda_i| (A_i \otimes e_i) (X \otimes I_k) (A_i^* \otimes e_i^*)=\sum_{i=1}^r |\lambda_i| (A_i X A_i^*) \otimes e_i e_i^*
\]
and so
\begin{align*}
  \Psi(X \otimes I_k) (I_n \otimes Q) &= \left( \sum_{i=1}^r |\lambda_i| (A_i X A_i^*) \otimes e_i e_i^* \right) \left( I_n \otimes \sum_{j=1}^r \text{sgn}(\lambda_j) e_j e_j^* \right)\\
  &= \sum_{i=1}^r |\lambda_i| \text{sgn}(\lambda_i) (A_i X A_i^*) \otimes e_i e_i^*,
\end{align*}
using the orthogonality of \(\{e_i\}\).
Taking the partial trace over \(\mathbb{M}_k\) we conclude
\begin{align*}
\text{tr}_k \left[\Psi(X \otimes I_k) (I_n \otimes Q)\right]&=
\text{tr}_k \left[ \sum_{i=1}^r |\lambda_i| \text{sgn}(\lambda_i) (A_i X A_i^*) \otimes e_i e_i^* \right] \\
&= \sum_{i=1}^r |\lambda_i| \text{sgn}(\lambda_i) (A_i X A_i^*) \text{tr}(e_i e_i^*) = \sum_{i=1}^r \lambda_i (A_i X A_i^*).
\end{align*}
On the other hand, we show that  \(\Phi(X) = \sum_{i=1}^r \lambda_i A_i X A_i^*\) and this will confirm  the representation \eqref{rep}.

The action of \(\Phi\) on any matrix \(X \in \mathbb{M}_m\) can be expressed as
\[
\Phi(X) = \text{tr}_m \left[ (X^T \otimes I_n) \mathbf{C}_\Phi \right],
\]
where \(\text{tr}_m\) denotes the partial trace over \(\mathbb{M}_m\), \(X^T\) is the transpose of \(X\)  (see Watrous \cite{watrous2018}). Substitute the spectral decomposition \(\mathbf{C}_\Phi = \sum_{i=1}^r \lambda_i u_i u_i^*\) we obtain
\begin{align}\label{hgt1}
\begin{split}
\Phi(X) &= \text{tr}_m \left[ (X^T \otimes I_n) \sum_{i=1}^r \lambda_i u_i u_i^* \right]\\
&=\sum_{i=1}^r \lambda_i \text{tr}_m \left[ (X^T \otimes I_n) u_i u_i^* \right]\\
&=\sum_{i=1}^r \lambda_i \text{tr}_m \left[ (X^T \otimes I_n) \mathrm{vec}(A_i) \mathrm{vec}(A_i)^*\right],
\end{split}
\end{align}
since $u_i=\mathrm{vec}(A_i)$.
Using the identity \((A \otimes B) \text{vec}(C) = \text{vec}(B C A^T)\) from Horn and Johnson \cite{horn1991}, we have
\[
(X^T \otimes I_n) \text{vec}(A_i) = \text{vec}(I_n A_i (X^T)^T) = \text{vec}(A_i X).
\]
Hence
\begin{align}\label{hgt2}
\text{tr}_m \left[ (X^T \otimes I_n) \mathrm{vec}(A_i) \mathrm{vec}(A_i)^*\right]=\text{tr}_m \left[  \text{vec}(A_i X)\mathrm{vec}(A_i)^*\right]=A_iXA_i^*,
\end{align}
in which the last equality follows from the fact that (see Bengtsson and Życzkowski \cite{bengtsson2006}) for \(B, C \in \mathbb{M}_{n \times m}\):
\[
\text{tr}_m \left[ \text{vec}(B) \text{vec}(C)^* \right] = B C^*.
\]
Combining \eqref{hgt1} and \eqref{hgt2}  we arrive at the desired representation for $\Phi$.

The dimension \(k = r\) is minimal, as it corresponds to the number of nonzero eigenvalues in the spectral decomposition of \(\mathbf{C}_{\Phi}\), ensuring the smallest possible auxiliary space.

\end{proof}

 Before giving some examples, we point out a  fact in the proof of Theorem~\ref{th3:main} which can be helpful in computations. We showed  that if \(\Phi: \mathbb{M}_m \rightarrow \mathbb{M}_n\) is Hermitian, then
 $$\Phi(X)=\sum_{j=1}^{nm}\lambda_i A_i X A_i^*,$$
in which $\lambda_i$'s are eigenvalues of the Choi matrix $\mathbf{C}_{\Phi} $ of $\Phi$ and $u_i=\text{vec}(A_i)$ are the corresponding eigenvectors.

We now present examples that illustrate the construction of the completely positive map $\Psi$ and the Hermitian matrix $Q$ for a given Hermitian map $\Phi$.
\begin{example}
Consider \(\Phi: \mathbb{M}_2 \rightarrow \mathbb{M}_2\) defined by \(\Phi(X) = X + X^*\). The Choi matrix is:
\[
\mathbf{C}_{\Phi} = \sum_{i,j=1}^2 E_{ij} \otimes (E_{ij} + E_{ji})=\begin{bmatrix}
2 & 0 & 0 & 1 \\
0 & 0 & 1 & 0 \\
0 & 1 & 0 & 0 \\
1 & 0 & 0 & 2
\end{bmatrix},
\]
where the standard basis ordering \(\{E_{11}, E_{12}, E_{21}, E_{22}\}\) for \(\mathbb{M}_4\)  is considered. Computing eigenvalues, we have \(\lambda = 3, 1, -1\) (with 1 having multiplicity 2). The corresponding eigenvectors are \\
 \(\frac{1}{\sqrt{2}}(1, 0, 0, 1)^T\) for \(\lambda = 3\);
\(\frac{1}{\sqrt{2}}(1, 0, 0, -1)^T\) and \(\frac{1}{\sqrt{2}}(0, 1, 1, 0)^T\) for  \(\lambda = 1\);  and
 \(\frac{1}{\sqrt{2}}(0, 1, -1, 0)^T\) for  \(\lambda = -1\).

Accordingly, we  define the operators \(A_i \in \mathbb{M}_2\) based on the eigenvectors:
\begin{align*}
A_1 &= \frac{1}{\sqrt{2}}(E_{11} + E_{22}), \qquad \qquad A_2 = \frac{1}{\sqrt{2}}(E_{11} - E_{22}),\\
A_3 &= \frac{1}{\sqrt{2}}(E_{12} + E_{21}),  \qquad \qquad A_4 = \frac{1}{\sqrt{2}}(E_{12} - E_{21}).
\end{align*}
WE define the completely positive map  \(\Psi: \mathbb{M}_2 \otimes \mathbb{M}_4 \to \mathbb{M}_2 \otimes \mathbb{M}_4\) using the eigenvalues as
\[
\Psi(Y) = 3 (A_1 \otimes e_1) Y (A_1^* \otimes e_1^*) + \sum_{i=2}^4 (A_i \otimes e_i) Y (A_i^* \otimes e_i^*),
\]
where \(\{e_i\}_{i=1}^4\) is the standard basis of \(\mathbb{C}^4\). Define the Hermitian matrix \(Q \in \mathbb{M}_4\):
\[
Q =e_1e_1^*+ e_2e_2^*+e_3e_3^*-e_4e_4^*=\text{diag}(1, 1, 1, -1),
\]
reflecting the signs of the eigenvalues.

\end{example}

When the Choi matrix \(\mathbf{C}_\Phi\) exhibits a block-diagonal structure, the Hermitian map can be decomposed into submaps, each potentially associated with a lower-rank Choi matrix, indicating that a smaller auxiliary space may be sufficient for the extension.

\begin{theorem}\label{1thm:reduced-cp-extension}
Let $\Phi: \mathbb{M}_m \to \mathbb{M}_n$ be a Hermitian linear map with Choi matrix $\mathbf{C}_\Phi \in \mathbb{M}_m \otimes \mathbb{M}_n$ of rank $r$. Suppose $\mathbf{C}_\Phi$ is block-diagonal and $k$ is the maximum rank of its diagonal blocks. Then there exist   a completely positive map $\Psi: \mathbb{M}_m \otimes \mathbb{M}_k \to \mathbb{M}_n \otimes \mathbb{M}_k$, and a Hermitian matrix $Q \in \mathbb{M}_k$ such that:
\[
\Phi(X) = \operatorname{tr}_k\left[\Psi(X \otimes I_k)(I_n \otimes Q)\right] \quad \forall X \in \mathbb{M}_m,
\]
provided the submaps $\Phi_i$ are compatible with a common auxiliary space of dimension $k$. The minimal dimension $k_{\min}$ satisfies $k_{\min} \leq r$, with equality if and only if $\mathbf{C}_\Phi$ is indecomposable (i.e., $b=1$ or $r_i = r$ for some $i$).
\end{theorem}

\begin{proof}
Since   $\mathbf{C}_\Phi \in \mathbb{M}_m \otimes \mathbb{M}_n$ is block-diagonal,  there exist orthogonal subspaces $\mathcal{H}_{m_i} \subseteq \mathbb{C}^m$ and $\mathcal{K}_{n_i} \subseteq \mathbb{C}^n$ such that $\mathbb{C}^m = \bigoplus_{i=1}^b \mathcal{H}_{m_i}$, $\mathbb{C}^n = \bigoplus_{i=1}^b \mathcal{K}_{n_i}$, with $\dim(\mathcal{H}_{m_i}) = m_i$, $\dim(\mathcal{K}_{n_i}) = n_i$, $\sum m_i = m$, and $\sum n_i = n$. The matrix algebra $\mathbb{M}_m$ decomposes as $\mathbb{M}_m \cong \bigoplus_{i=1}^b \mathbb{M}_{m_i}$, and similarly $\mathbb{M}_n \cong \bigoplus_{i=1}^b \mathbb{M}_{n_i}$. The Choi matrix is:
\[
\mathbf{C}_\Phi = \bigoplus_{i=1}^b \mathbf{C}_{\Phi_i}, \quad \mathbf{C}_{\Phi_i} \in \mathbb{M}_{m_i} \otimes \mathbb{M}_{n_i}.
\]
 Each $\Phi_i: \mathbb{M}_{m_i} \to \mathbb{M}_{n_i}$ is a Hermitian linear map and the rank of $\mathbf{C}_\Phi$ is $r = \sum_{i=1}^b r_i$, where $r_i = \operatorname{rank}(\mathbf{C}_{\Phi_i})$.

We apply Theorem~\ref{th3:main} for  each submap $\Phi_i$. Since $\mathbf{C}_{\Phi_i}$ is Hermitian with rank $r_i$, there exists a completely positive map $\Psi_i: \mathbb{M}_{m_i} \otimes \mathbb{M}_{r_i} \to \mathbb{M}_{n_i} \otimes \mathbb{M}_{r_i}$ and a Hermitian matrix $Q_i \in \mathbb{M}_{r_i}$ such that:
\[
\Phi_i(X_i) = \operatorname{tr}_{r_i}\left[\Psi_i(X_i \otimes I_{r_i})(I_{n_i} \otimes Q_i)\right] \quad \forall X_i \in \mathbb{M}_{m_i},\,\,(i=1,\dots,b).
\]
The spectral decomposition of $\mathbf{C}_{\Phi_i}$ is:
\[
\mathbf{C}_{\Phi_i} = \sum_{j=1}^{r_i} \lambda_{i,j} u_{i,j} u_{i,j}^*,
\]
where $\lambda_{i,j} \in \mathbb{R}$, $\{u_{i,j}\}_{j=1}^{r_i}$ are orthonormal vectors in $\mathbb{C}^{m_i n_i}$, and $\Psi_i$ is defined via Kraus operators $A_{i,j}: \mathbb{C}^{m_i} \to \mathbb{C}^{n_i}$ with $\operatorname{vec}(A_{i,j}) = u_{i,j}$, as:
\[
\Psi_i(Y_i) = \sum_{j=1}^{r_i} |\lambda_{i,j}| (A_{i,j} \otimes e_{i,j}) Y_i (A_{i,j}^* \otimes e_{i,j}^*),
\]
where $\{e_{i,j}\}_{j=1}^{r_i}$ is an orthonormal basis for $\mathbb{C}^{r_i}$. The matrix $Q_i = \sum_{j=1}^{r_i} \operatorname{sgn}(\lambda_{i,j}) e_{i,j} e_{i,j}^*$ is Hermitian.

Set $k = \max\{r_i : i = 1, \ldots, b\}$, the largest rank among the    Choi matrices of submaps $\Phi_i$. Let $\mathcal{K} = \mathbb{C}^k$ be the auxiliary Hilbert space with orthonormal basis $\{e_l\}_{l=1}^k$. For each $i$, define an isometry $V_i: \mathbb{C}^{r_i} \to \mathbb{C}^k$ by $V_i e_{i,j} = e_j$ for $j = 1, \ldots, r_i$, and extend to the remaining basis vectors arbitrarily (e.g., $V_i e_{i,j} = e_j$ for $j > r_i$ if $r_i < k$). The adjoint is $V_i^*: \mathbb{C}^k \to \mathbb{C}^{r_i}$, with $V_i^* V_i = I_{r_i}$.

Define     the   map $\Psi: \mathbb{M}_m \otimes \mathbb{M}_k \to \mathbb{M}_n \otimes \mathbb{M}_k$ by
\[
\Psi(Y) = \sum_{i=1}^b \sum_{j=1}^{r_i} |\lambda_{i,j}| (A_{i,j} \otimes V_i e_{i,j}) Y (A_{i,j}^* \otimes (V_i e_{i,j})^*).
\]
Since $V_i e_{i,j} = e_j$, we have $(A_{i,j} \otimes V_i e_{i,j}) = (A_{i,j} \otimes e_j)$, and the map is Completely positive, because it is a sum of Kraus operators. For $Y = X \otimes I_k$, compute:
\[
\Psi(X \otimes I_k) = \bigoplus_{i=1}^b \sum_{j=1}^{r_i} |\lambda_{i,j}| (A_{i,j} X_i A_{i,j}^*) \otimes e_j e_j^*,
\]
since $(A_{i,j} \otimes e_j)(X \otimes I_k) = (A_{i,j} X_i) \otimes e_j$.
Define $Q \in \mathbb{M}_k$  by
\[
Q = \sum_{i=1}^b V_i Q_i V_i^*,
\]
where $Q_i = \sum_{j=1}^{r_i} \operatorname{sgn}(\lambda_{i,j}) e_{i,j} e_{i,j}^*$. Since each $Q_i$ is Hermitian, and $V_i^* V_i = I_{r_i}$, $Q$ is a sum of Hermitian operators, hence Hermitian. Explicitly:
\[
V_i Q_i V_i^* = \sum_{j=1}^{r_i} \operatorname{sgn}(\lambda_{i,j}) V_i e_{i,j} (V_i e_{i,j})^* = \sum_{j=1}^{r_i} \operatorname{sgn}(\lambda_{i,j}) e_j e_j^*.
\]
Thus, $Q = \sum_{i=1}^b \sum_{j=1}^{r_i} \operatorname{sgn}(\lambda_{i,j}) e_j e_j^*\), where the sum over $j$ is restricted to the indices used by each block.

We have
\begin{align*}
\Psi(X \otimes I_k)(I_n \otimes Q) &= \bigoplus_{i=1}^b \sum_{j=1}^{r_i} |\lambda_{i,j}| (A_{i,j} X_i A_{i,j}^*) \otimes e_j e_j^* \cdot \bigoplus_{i'=1}^b \sum_{j'=1}^{r_{i'}} \operatorname{sgn}(\lambda_{i',j'}) I_{n_i} \otimes e_{j'} e_{j'}^*\\
&=\bigoplus_{i=1}^b \sum_{j=1}^{r_i} |\lambda_{i,j}| \operatorname{sgn}(\lambda_{i,j}) (A_{i,j} X_i A_{i,j}^*) \otimes e_j e_j^*.
\end{align*}

Taking the partial trace over $\mathbb{M}_k$ we get
\[
\operatorname{tr}_k\left[\Psi(X \otimes I_k)(I_n \otimes Q)\right] = \bigoplus_{i=1}^b \sum_{j=1}^{r_i} \lambda_{i,j} (A_{i,j} X_i A_{i,j}^*) = \bigoplus_{i=1}^b \Phi_i(X_i) = \Phi(X),
\]
since $\Phi(X) = \bigoplus_{i=1}^b \Phi_i(X_i)$ and $\Phi_i(X_i) = \sum_{j=1}^{r_i} \lambda_{i,j} A_{i,j} X_i A_{i,j}^*$  as in Theorem~\ref{th3:main}.

\end{proof}

\begin{example}\label{ex_reduc}
Consider the Hermitian linear map $\Phi: \mathbb{M}_4 \to \mathbb{M}_4$ defined by $\Phi(X) = \bigoplus_{i=1}^2 \Phi_i(X_i)$, where $X = \begin{bmatrix} X_1 & 0 \\ 0 & X_2 \end{bmatrix} \in \mathbb{M}_4$, $X_i \in \mathbb{M}_2$, and the submaps are:
\[
\Phi_1(X_1) = X_1 + X_1^*, \quad \Phi_2(X_2) = X_2 - X_2^*,
\]
We compute each part of Theorem~\ref{1thm:reduced-cp-extension} to construct a CP extension with auxiliary dimension $k = 2 < r = 4$.
The input and output spaces decompose as $\mathbb{M}_4 \cong \mathbb{M}_2 \oplus \mathbb{M}_2$, with $m_1 = m_2 = n_1 = n_2 = 2$. The Choi matrix of $\Phi$ is:
\[
\mathbf{C}_\Phi = \sum_{i,j=1}^4 E_{ij} \otimes \Phi(E_{ij}),
\]
where $E_{ij} $ are matrix units in $\mathbb{M}_4$. Since $\Phi(X) = \bigoplus_{i=1}^2 \Phi_i(X_i)$, we compute $\mathbf{C}_\Phi = \bigoplus_{i=1}^2 \mathbf{C}_{\Phi_i}$, where:
\[
\mathbf{C}_{\Phi_i} = \sum_{j,k=1}^2 E_{jk}^{(i)} \otimes \Phi_i(E_{jk}^{(i)}),
\]
and $E_{jk}^{(i)}$ are matrix units in $\mathbb{M}_2$ for block $i$. For $\Phi_1(X_1) = X_1 + X_1^*$ we have
\[
\mathbf{C}_{\Phi_1} = \begin{bmatrix} 2 & 0 & 0 & 1 \\ 0 & 0 & 1 & 0 \\ 0 & 1 & 0 & 0 \\ 1 & 0 & 0 & 2 \end{bmatrix},
\]
as in Example 2.8. For $\Phi_2(X_2) = X_2 - X_2^*$ we have
\[
\mathbf{C}_{\Phi_2} = \begin{bmatrix} 0 & 0 & 0 & 0 \\ 0 & 0 & 1 & 0 \\ 0 & -1 & 0 & 0 \\ 0 & 0 & 0 & 0 \end{bmatrix}.
\]
The full Choi matrix is block-diagonal
$ \mathbf{C}_\Phi =  \mathbf{C}_{\Phi_1} \oplus \mathbf{C}_{\Phi_2}$.
Compute ranks: For \(\mathbf{C}_{\Phi_1}\), eigenvalues are 3, 1, 1, -1 (Example 2.8), so \(r_1 = 4\). For \(\mathbf{C}_{\Phi_2}\), eigenvalues are \(\pm 1\) (each multiplicity 1) and 0 (multiplicity 2), so \(r_2 = 2\). Total rank: \(r = r_1 + r_2 = 4 + 2 = 6\).

For \(\Phi_1\), use Example 2.8’s spectral decomposition:
\[
\mathbf{C}_{\Phi_1} = 3 u_1 u_1^* + u_2 u_2^* + u_3 u_3^* - u_4 u_4^*,
\]
with eigenvectors:
\[
u_1 = \frac{1}{\sqrt{2}}\begin{bmatrix} 1 \\ 0 \\ 0 \\ 1 \end{bmatrix}, \quad u_2 = \frac{1}{\sqrt{2}}\begin{bmatrix} 1 \\ 0 \\ 0 \\ -1 \end{bmatrix}, \quad u_3 = \frac{1}{\sqrt{2}}\begin{bmatrix} 0 \\ 1 \\ 1 \\ 0 \end{bmatrix}, \quad u_4 = \frac{1}{\sqrt{2}}\begin{bmatrix} 0 \\ 1 \\ -1 \\ 0 \end{bmatrix}.
\]
Corresponding operators \(A_{1,j} \in \mathbb{M}_2\) (via \(\operatorname{vec}(A_{1,j}) = u_j\)):
\[
A_{1,1} = \frac{1}{\sqrt{2}}\begin{bmatrix} 1 & 0 \\ 0 & 1 \end{bmatrix}, \quad A_{1,2} = \frac{1}{\sqrt{2}}\begin{bmatrix} 1 & 0 \\ 0 & -1 \end{bmatrix}, \quad A_{1,3} = \frac{1}{\sqrt{2}}\begin{bmatrix} 0 & 1 \\ 1 & 0 \end{bmatrix}, \quad A_{1,4} = \frac{1}{\sqrt{2}}\begin{bmatrix} 0 & 1 \\ -1 & 0 \end{bmatrix}.
\]
Define \(\Psi_1: \mathbb{M}_2 \otimes \mathbb{M}_4 \to \mathbb{M}_2 \otimes \mathbb{M}_4\):
\[
\Psi_1(Y_1) = 3 (A_{1,1} \otimes e_{1,1}) Y_1 (A_{1,1}^* \otimes e_{1,1}^*) + \sum_{j=2}^4 (A_{1,j} \otimes e_{1,j}) Y_1 (A_{1,j}^* \otimes e_{1,j}^*),
\]
with basis \(\{e_{1,j}\}_{j=1}^4\) for \(\mathbb{C}^4\), and:
\[
Q_1 = e_{1,1} e_{1,1}^* + e_{1,2} e_{1,2}^* + e_{1,3} e_{1,3}^* - e_{1,4} e_{1,4}^* = \operatorname{diag}(1, 1, 1, -1).
\]
For \(\Phi_2\), compute the spectral decomposition of \(\mathbf{C}_{\Phi_2}\):
\[
\mathbf{C}_{\Phi_2} = \begin{bmatrix} 0 & 0 & 0 & 0 \\ 0 & 0 & 1 & 0 \\ 0 & -1 & 0 & 0 \\ 0 & 0 & 0 & 0 \end{bmatrix},
\]
with eigenvalues 1, -1 (each multiplicity 1), and eigenvectors:
\[
u_{2,1} = \frac{1}{\sqrt{2}}\begin{bmatrix} 0 \\ 1 \\ 1 \\ 0 \end{bmatrix} (\lambda = 1), \quad u_{2,2} = \frac{1}{\sqrt{2}}\begin{bmatrix} 0 \\ 1 \\ -1 \\ 0 \end{bmatrix} (\lambda = -1).
\]
Operators: \(A_{2,1} = \frac{1}{\sqrt{2}}\begin{bmatrix} 0 & 1 \\ 1 & 0 \end{bmatrix}\), \(A_{2,2} = \frac{1}{\sqrt{2}}\begin{bmatrix} 0 & 1 \\ -1 & 0 \end{bmatrix}\). Define \(\Psi_2: \mathbb{M}_2 \otimes \mathbb{M}_2 \to \mathbb{M}_2 \otimes \mathbb{M}_2\):
\[
\Psi_2(Y_2) = (A_{2,1} \otimes e_{2,1}) Y_2 (A_{2,1}^* \otimes e_{2,1}^*) + (A_{2,2} \otimes e_{2,2}) Y_2 (A_{2,2}^* \otimes e_{2,2}^*),
\]
with \(\{e_{2,j}\}_{j=1}^2\) for \(\mathbb{C}^2\), and:
\[
Q_2 = e_{2,1} e_{2,1}^* - e_{2,2} e_{2,2}^* = \begin{bmatrix} 1 & 0 \\ 0 & -1 \end{bmatrix}.
\]

Set $k = \max\{r_1, r_2\} = \max\{4, 2\} = 4$. Let \(\mathcal{K} = \mathbb{C}^4\) with basis \(\{e_l\}_{l=1}^4\). Define isometries: for \(\Phi_1\), \(r_1 = 4\), so $V_1: \mathbb{C}^4 \to \mathbb{C}^4\), $V_1 e_{1,j} = e_j$, i.e., $V_1 = I_4$. For \(\Phi_2\), $r_2 = 2$, so $V_2: \mathbb{C}^2 \to \mathbb{C}^4\), $V_2 e_{2,j} = e_j$ for $j=1,2$, represented as:
\[
V_2 = \begin{bmatrix} 1 & 0 \\ 0 & 1 \\ 0 & 0 \\ 0 & 0 \end{bmatrix}.
\]
Verify: $V_2^* V_2 = I_2$, $V_1^* V_1 = I_4$.

Define $\Psi: \mathbb{M}_4 \otimes \mathbb{M}_4 \to \mathbb{M}_4 \otimes \mathbb{M}_4$:
\[
\Psi(Y) = 3 (A_{1,1} \otimes e_1) Y (A_{1,1}^* \otimes e_1^*) + \sum_{j=2}^4 (A_{1,j} \otimes e_j) Y (A_{1,j}^* \otimes e_j^*) + \sum_{j=1}^2 (A_{2,j} \otimes e_j) Y (A_{2,j}^* \otimes e_j^*).
\]
For $Y = X \otimes I_4$, $X = \begin{bmatrix} X_1 & 0 \\ 0 & X_2 \end{bmatrix}$:
\[
\Psi(X \otimes I_4) = \begin{bmatrix} 3 A_{1,1} X_1 A_{1,1}^* + \sum_{j=2}^4 A_{1,j} X_1 A_{1,j}^* & 0 \\ 0 & \sum_{j=1}^2 A_{2,j} X_2 A_{2,j}^* \end{bmatrix} \otimes \text{diagonal terms}.
\]

We have
\[
Q_1 = \operatorname{diag}(1, 1, 1, -1), \quad V_1 Q_1 V_1^* = Q_1.
\]
\[
V_2 Q_2 V_2^* = \begin{bmatrix} 1 & 0 \\ 0 & 1 \\ 0 & 0 \\ 0 & 0 \end{bmatrix} \begin{bmatrix} 1 & 0 \\ 0 & -1 \end{bmatrix} \begin{bmatrix} 1 & 0 & 0 & 0 \\ 0 & 1 & 0 & 0 \end{bmatrix} = \begin{bmatrix} 1 & 0 & 0 & 0 \\ 0 & -1 & 0 & 0 \\ 0 & 0 & 0 & 0 \\ 0 & 0 & 0 & 0 \end{bmatrix}.
\]
\[
Q = V_1 Q_1 V_1^* + V_2 Q_2 V_2^* = \begin{bmatrix} 1 & 0 & 0 & 0 \\ 0 & 1 & 0 & 0 \\ 0 & 0 & 1 & 0 \\ 0 & 0 & 0 & -1 \end{bmatrix} + \begin{bmatrix} 1 & 0 & 0 & 0 \\ 0 & -1 & 0 & 0 \\ 0 & 0 & 0 & 0 \\ 0 & 0 & 0 & 0 \end{bmatrix} = \begin{bmatrix} 2 & 0 & 0 & 0 \\ 0 & 0 & 0 & 0 \\ 0 & 0 & 1 & 0 \\ 0 & 0 & 0 & -1 \end{bmatrix}.
\]

 \end{example}



\begin{thebibliography}{9}
\bibitem{ando2018}
T. Ando, \textit{Positive map as difference of two completely positive or super-positive maps}, Adv. Opr. Theory, {\bf 3}(1) (2018), 53--60.


\bibitem{bengtsson2006}
I. Bengtsson,  and K.  \.{Z}yczkowski,    \textit{Geometry of Quantum States: An introduction to quantum entanglement},  Cambridge University Press, 2006.

\bibitem{buscemi2013}
F. Buscemi, M. Dall'Arno, M. Ozawa, and V. Vedral, \textit{Direct observation of any two-point quantum correlation function},  arXiv:1312.4240v1, (2013).

\bibitem{choi1975} 
M. -D. Choi,  \textit{Completely positive linear maps on complex matrices} Linear Algebra Appl., {\bf 10} (3) (1975), 285--290.

\bibitem{dadkhah2024}
A. Dadkhah, M. Kian, and M. S. Moslehian, \textit{Decomposition of tracial positive maps and applications in quantum information},  Anal. Math. Phys., {\bf 14} (2024), 48.

\bibitem{han2024}
K. H. Han, S.-H. Kye, and E. St\o rmer, \textit{Infinite dimensional analogues of Choi matrices}, J. Funct. Anal, {\bf 287} (8) (2024), 110557.

\bibitem{homa2024} 
G. Homa, A. Ortega, M. Koniorczyk, \textit{Choi Representation of Completely Positive Maps in Brief},  Z.  Naturforsch.,  {\bf 79} (12) (2024), 1123--1133.

\bibitem{horn1991} R. A. Horn,   and C. R.  Johnson,    \textit{Topics in matrix analysis}, Cambridge University Press, 1991.

\bibitem{kian2025}
M. Kian, \textit{Deviation from complete positivity: Structural insights and quantum information applications}, arXiv:2506.03773, (2025).

\bibitem{Majewski2024} W. A. Majewski, \textit{On Positive Maps in Quantum Information}, Russian J. Math. Phys. {\bf  21} (2014), 362--372.

    \bibitem{parzygnat2023}
A. J. Parzygnat, J. Fullwood, F. Buscemi, and G. Chiribella, \textit{Virtual quantum broadcasting}, Phys. Rev. Lett. {\bf 132}  (2024), 110203.

\bibitem{Piotrowski2023} J. Piotrowski, L. Skowronek, \textit{New Tools for Investigating Positive Maps in Matrix Algebras}, Linear Algebra Appl., {\bf 662} (2023), 39--48.

 \bibitem{sisu2018}
 K. Siudzi\'{n}ska, S.  Chakraborty,  and D.  Chru\'{s}ci\'{n}ski,   \textit{Interpolating between Positive and Completely Positive Maps: A New Hierarchy of Entangled States},  Linear Algebra Appl.,   {\bf 662} (2018), 39--48

 \bibitem{stormer2013}
E. St\o rmer, \textit{Positive linear maps of operator algebras}, Springer, 2013.

\bibitem{watrous2018} J. Watrous,   \textit{The theory of quantum information}, Cambridge University Press, 2018.

\end{thebibliography}
\end{document}